\documentclass[12pt, reqno]{amsart}

\usepackage{amsmath}
\usepackage{amssymb}
\usepackage{amsthm}
\usepackage[enableskew]{youngtab}
\usepackage{epic}
\usepackage{eepic}
\usepackage{ecltree}
\usepackage{color,graphicx,amsthm,amscd,verbatim}

\newcommand{\note}{\medskip\noindent{\bf Note: \ }}

\newcommand{\ds}{\displaystyle}

\newtheorem{theorem}{Theorem}[section]
\newtheorem{prop}[theorem]{Proposition}
\newtheorem{lemma}[theorem]{Lemma}
\newtheorem{cor}[theorem]{Corollary}
\newtheorem{con}[theorem]{Conjecture}
\newtheorem{definition1}[theorem]{Definition}

\def\l{\lambda}
\def\m{\mu}
\def\n{\nu}
\def\a{\alpha}
\def\b{\beta}
\def\d{\delta}
\def\g{\gamma}

\begin{document}

\title[Powers of the Vandermonde determinant: recursions]{Powers of the Vandermonde determinant, Schur Functions, and recursive formulas}
\author{C. Ballantine} 
\address{Department of Mathematics and Computer Science, College of
  the Holy Cross, Worcester, MA 01610}
\email{cballant@holycross.edu}

\date{\today}
\begin{abstract}

Since every even power of the Vandermonde determinant is a symmetric polynomial, we want to understand its decomposition in terms of the basis of Schur functions. We investigate  several combinatorial properties of the coefficients in the decomposition. In particular, we give recursive formulas for the coefficient of the Schur function $s_{\m}$ in the decomposition of an even power of the Vandermonde determinant in $n + 1$ variables in terms of the coefficient of the Schur function $s_{\l}$ in the decomposition of the same even power of the Vandermonde determinant in $n$ variables if the Young diagram of $\m$ is obtained from the Young diagram of $\l$ by adding a tetris type shape to the top or to the left. An extended abstract containing the statement of the results presented here appeared in the Proceedings of FPSAC11 \cite{B}.\medskip

\noindent Mathematics Standard Classification: 05E05, 15A15 

\noindent Keywords: Vandermonde determinant, Schur functions, Quantum Hall effect

\end{abstract}

\maketitle

\section{ Introduction}
\noindent 

 In the theory of symmetric functions Vandermonde determinants are best known for the part they play in the classical definition of Schur functions. Since each even power of the Vandermonde determinant is a symmetric function, it is natural to ask for its decomposition in terms
of the basis for the ring of symmetric functions given by Schur functions \cite{S}. This decomposition has been studied extensively (see \cite{D}, \cite{FGIL}  and the references therein) in connection with its usefulness in the understanding of the (fractional) quantum Hall effect. In particular, the coefficients in the decomposition correspond precisely to the coefficients in the decomposition of the Laughlin wave function as a linear combination of (normalized) Slater determinantal wave functions. The calculation of the coefficients in the decomposition becomes computationally expensive as the size of the determinant increases. Several algorithms for the expansion of the square of the Vandermonde determinant in terms of Schur functions are available (see, for example \cite{STW}). However, a combinatorial interpretation for the coefficient of a given Schur function is still unknown. Recently, Boussicault, Luque and Tollu \cite{BLT} provided a purely numerical algorithm for computing the coefficient of a given Schur function in the decomposition without computing the other coefficients. The algorithm uses hyperdeterminants and their Laplace expansion. It was used by the authors to compute coefficients in the decomposition of even powers of the Vandermonde determinant of size up to $11$. For determinants of large size, the algorithm becomes computationally too expensive for practical purposes. In this article we present recursive combinatorial properties of some of the coefficients in the decomposition. Specifically, the coefficient of the Schur function $s_{\m}$ in the decomposition of an even power of the Vandermonde determinant in $n + 1$ variables  is computed in terms of the coefficient of the Schur function $s_{\l}$ in the decomposition of the same even power of the Vandermonde determinant in $n$ variables if the Young diagram of $\m$ is obtained from the Young diagram of $\l$ by adding a tetris type shape to the top or to the left.

In section 2 we introduce the notation and basic facts about partitions and Schur functions and their relation to the Vandermonde determinant. In section 3 we give an elementary proof of the fact that the Schur function corresponding to a partition $\l$ and that corresponding to the  reverse partition $\l^{bc}$ (as defined by \cite{D}) have the same coefficient in the decomposition of the (correct) even power of the Vandermonde determinant. In section 4 we exhibit two simple  recursion rules followed in section 5 by two new and somewhat surprising recursive formulas. In section 5 we also present a third, conjectural, formula which  has been verified for $n\leq 6$ using Maple. We prove  two  special cases of this formula.  In section 6 we use the recursive formulas of sections 4 and 5 to prove several closed formulas and recursive observations given in \cite{D}, one of the pioneering articles in using the decomposition of the square of the Vandermonde determinant in terms of Schur functions to understand the quantum Hall effect. Our results improve considerably on the observations in \cite{D}. 

\section{Notation and basic facts}

\noindent We first introduce some notation and basic facts about the Vandermonde determinant related to this problem. For details on partitions and Schur functions we refer the reader to [6, Chapter 7].

Let $n$ be a non-negative integer. A \emph{partition} of $n$ is a weakly decreasing sequence of non-negative integers, $\l := (\l_1,\l_2, \ldots, \l_{\ell})$, such that $|\l|:= \sum \l_i=n$. We write $\l \vdash n$ to mean $\l$ is a partition of $n$. The  integers $\l_i$ are called the parts of $\l$. We identify a partition with its \emph{Young diagram}, i.e. the array of left-justified squares (boxes) with $\l_1$ boxes in the first row, $\l_2$ boxes in the second row, and so on. The rows are arranged in matrix form from top to bottom. By the box in position $(i,j)$ we mean the box in the $i$-th row and $j$-th column of $\l$. The \emph{length} of $\l$, $\ell(\l)$, is the number of rows in the Young diagram or the number of non-zero parts of $\l$.
For example,
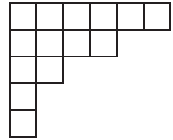
\vspace{.1in}

\begin{center}\scriptsize$\yng(6,4,2,1,1)$ \end{center}
is the Young diagram for $\l = (6,4,2,1,1)$, with $\ell(\l) = 5$ and $|\l| = 14$. 

We write $\l=\langle 1^{m_1},2^{m_2} \ldots \rangle$ to mean that $\l$ has $m_i$ parts equal to $i$. 

Given a weak composition $\a = (\a_1,\a_2,\ldots, \a_n)$ of length $n$, we write $x^{\a}$
for the monomial $x^{\a_1}x^{\a_2} \cdots x^{\a_n}$. If $\l := (\l_1,\l_2, \ldots, \l_n)$ is a partition of length at most $n$ and $ \d_n = (n - 1,n - 2, \ldots ,2,1,0)$, then the skew symmetric function
$\ds a_{\l+\d_n}$ is defined as \begin{equation}a_{\l+\d_n}=\det (x_i^{\l_j+n-j})_{i,j=1}^n.\label{skewfct}\end{equation}
 If $\l=\emptyset$,
\begin{equation} a_{\d_n} = \det(x_i^{n-j})_{i,j=1}^n= \prod_{1\leq i<j\leq n}(x_i-x_j) \end{equation}is the Vandermonde determinant. We have \cite[Theorem 7.15.1]{S}
\begin{equation} \frac{a_{\l+\d_n}}{a_{\d_n}} =s_{\l}(x_1,\ldots ,x_n),	\end{equation}
where $s_{\l}(x_1, \ldots, x_n)$ is the Schur function of shape $\l$ in variables  $x_1, \ldots , $ $x_n$. Moreover, if we denote by $[x^{\l+\d_n}]a_{\d_n}f$ the coefficient of $x^{\l+\d_n}$ in $a_{\d_n}f$, then \cite[Corollary 7.15.2]{S} for any homogeneous symmetric function $f$ of degree $n$, the coefficient of $s_{\l}$ in the decomposition of $f$ is given by  
\begin{equation}\langle f, s_{\l} \rangle = [x^{\l+\d_n}]a_{\d_n}f.\end{equation} In particular, if $f = a_{\d_n}^{2k}$, then
\begin{equation}\langle a_{\d_n}^{2k}, s_{\l}\rangle = [x^{\l+\d_n}]a_{\d_n}^{2k+1}.\label{schurinvand}	\end{equation}
We will often write $c_{\l}$ for $\langle a_{\d_n}^{2k}, s_{\l}\rangle$.

The goal of this work is to investigate several combinatorial properties of the numbers (\ref{schurinvand}).

The following proposition summarizes some easy to prove properties that are frequently used in the article. 

\begin{prop} We have 
\begin{itemize}
\item[(i)] The size of $\d_n$ is given by $|\d_n| = n(n-1)/2$

\item[(ii)] The skew symmetric function $a_{\d_n}$ is a homogeneous polynomial of degree $n(n-1)/2$. 

\item[(iii)] If $\langle a_{\d_n}^{2k},s_{\l}\rangle \not =0$, then $|\l| = kn(n-1)$, $n-1\leq \ell(\l)\leq n$,
$ k(n-1) \leq \l_1 \leq 2k(n-1)$ and $\l_n \leq k(n-1)$. 

\item[(iv)] Moreover, if $\l_n = k(n-1)$ in \emph{(iii)},
then $\l = \langle(k(n-1))^n\rangle$.\end{itemize}\end{prop}


By $\bar{a}_{\d_n}$	we mean $a_{\d_n}$ with $x_i$ replaced by $x_{i+1}$ for each $i = 1,2,\ldots, n$. Thus, \begin{equation} \bar{a}_{\d_n}= \prod_{2\leq i <j \leq n+1} (x_i-x_j).	\end{equation}
By $\bar{x}^{\a}$, where $\a$ is the weak composition $\a= (\a_1,\a_2,\ldots,\a_n)$, we mean $x^{\a}$
with $x_i$ replaced by $x_{i+1}$ for each $i = 1,2,\ldots n$. Thus, \begin{equation} \bar{x}^{\a}=x_2^{\a_1}x_3^{\a_2}\cdots x_{n+1}^{\a_n} .	\end{equation}

 Given a weak composition $\a= (\a_1,\a_2,\ldots,\a_n)$ of $n$ of length at most
$n$, we denote by $c_{\a}$ the coefficient of $x^{\a}$ in $a_{\d_n}^{2k+1}$. If $\xi$ is a permutation of $\{1, 2, \ldots, n\}$, and $\xi(\a)$ is the weak composition $(\a_{\xi(1)},\a_{\xi(2)}, \ldots, \a_{\xi(n)})$,  one can easily see that
\begin{equation} c_{\a} = sgn(\xi)c_{\xi(\a)}.\label{calpha}	\end{equation} 

\section{The box-complement of a partition}

\begin{definition1} Let  $\l=(\l_1,\l_2,\ldots, \l_{\ell(\l)})$ be a partition of $kn(n-1)$ with $\ell(\l)\leq n$. The \emph{box-complement} of $\l$ is the partition of $kn(n-1)$ given by \begin{equation}\l^{bc}=(2k(n-1)-\l_n, 2k(n-1)-\l_{n-1}, \ldots,2k(n-1)-\l_1).\label{bc} \end{equation} 
\end{definition1}
\vspace{.1in}

Thus, $\l^{bc}$ is obtained from $\l$ in the following way. Place the Young diagram of $\l$ in the upper left corner of a box with $n$ rows each of length $2k(n-1)$. If we remove the Young diagram of $\l$ and rotate the remaining shape by $180^{\circ}$, we obtain the Young diagram of $\l^{bc}$. \vspace{.1in}

\noindent {\bf Example:} Let $k=1$, $n=4$ and $\l=(5,3,2,2)$. Then $\l^{bc}=(4,4,3,2)$. The Young diagram of $\l$ is shown on the left of the $4\times 6$ box. The remaining squares of the box are marked with $X$. They form the diagram of $\l^{bc}$ rotated by $180^{\circ}$. \begin{center}    \young(\hfil\hfil\hfil\hfil\hfil{X},\hfil\hfil\hfil XXX,\hfil\hfil XXXX,\hfil\hfil XXXX)\end{center}\vspace{.1in}

\begin{lemma}\label{bclemma} {\bf (Box-complement lemma)} With the notation above, we have  \begin{equation}\langle a_{\delta}^{2k},s_{\l}\rangle =\langle a_{\delta}^{2k},s_{\l^{bc}}\rangle\end{equation}

\end{lemma}

For a proof in the case $k = 1$, see \cite[Section 6]{D} where the box-complement partition is referred to as the reversed partition. We prove the lemma for general $k$ by elementary means, using induction on $n$. In \cite{D}, Dunne also explains the physical meaning of the box-complement lemma.

\begin{proof} We use induction on $n$. If $n = 1$, $a_{\d_1} = 1$ and the only partition $\l$
for which $\langle a_{\d_1},s_{\l}\rangle \not = 0$ is the empty partition. Its box-complement is also
the empty partition. If $n = 2$, $a_{\d_2}^{2k+1} = (x_1 - x_2)^{2k+1}$. A partition $\l$ of $2k$
with $\ell(\l) \leq 2$ for which $\langle a_{\d_2},s_{\l}\rangle \not = 0$	 is of the form $\l = (\l_1,2k-\l_1)$, with $k\leq \l_1\leq 2k$. The box-complement of $\l$ is $\l^{bc} = (2k-(2k-\l_1),2k-\l_1) = \l$. Thus, each contributing partition is its own box-complement.
\vspace{.1in}

For the induction step, assume that $\ds \langle a_{\d_{n-1}}^{2k},s_{\m}\rangle =\langle a_{\d_{n-1}}^{2k},s_{\m^{bc}}\rangle$, for all partitions $\m$ of $k(n-1)(n-2)$ with $\ell(\m)\leq n-1$. Note that (\ref{calpha}) implies that the statement of the lemma is true for all weak compositions of $k(n-1)(n-2)$,
not just for partitions.

Fix $\l=(\l_1,\l_2,\ldots, \l_{n-1},\l_n)\vdash kn(n-1)$ with $n-1 \leq \ell(\l) \leq n$ (we
allow $\l_n = 0$). We will set up a bijective correspondence between terms
in the expansion of $a_{\d_n}^{2k+1}$ which are multiples of $x^{\l+\d_n}$ and terms which are multiples of   $x^{\l^{bc}+\d_n}$.

We first write $\ds a_{\d_n}^{2k+1}$ as \begin{equation}\ds a_{\d_n}^{2k+1}=a_{\d_{n-1}}^{2k+1}\ \prod_{i=1}^{n-1}(x_i-x_n)^{2k+1}= a_{\d_{n-1}}^{2k+1} \cdot P_1.\end{equation}
The product $P_1$ is the only part of $\ds a_{\d_{n}}^{2k+1}$ contributing powers of $x_n$ to $x^{\l+\d_n}$.

Consider a weak composition \begin{equation}\a=(\a_1,\a_2, \ldots, \a_{n-1})\label{weakcomp}\end{equation}
of $(2k+1)(n-1)-\l_n$ with $0 \leq \a_i \leq \min(2k+1,\l_i+n-i)$, $i = 1,2,\ldots ,n-1$.
Suppose
\begin{equation} \label{alpha} \ds x^{\a}x_n^{\l_n}=x_1^{\a_1}x_2^{\a_2}\cdots x_{n-1}^{\a_{n-1}}x_n^{\l_n}\end{equation}
appears in $P_1$ with coefficient $p_{\a}$ and \begin{equation}\frac{x^{\l+\d_n}}{x^{\a}x_n^{\l_n}}=:x^{\n+\d_{n-1}}\end{equation} appears in $\ds a_{\d_{n-1}}^{2k+1}$ with coefficient $s_{\a}$. Here, $\n = (\n_1,\n_2,\ldots,\n_{n-1})$ is a
weak composition of $k(n - 1)(n - 2)$ with $$\n_i = \l_i-\a_i+ 1.$$
Now we write $\ds a_{\d_n}^{2k+1}$ as \begin{equation}\ds a_{\d_n}^{2k+1}= \bar{a}_{\d_{n-1}}^{2k+1} \prod_{i=2}^{n}(x_1-x_i)^{2k+1}=\bar{a}_{\d_{n-1}}^{2k+1}\cdot P_2.\end{equation} The product $P_2$ is the only part of $a_{\d_n}^{2k+1}$ contributing powers of $x_1$ to $x^{\l^{bc}+\d_n}$.
The weak composition $\a$ in (\ref{weakcomp}) uniquely determines the  weak composition $\tilde{\a}=(2k+1-\a_{n-1},2k+1-\a_{n-2}, \ldots, 2k+1-\a_{1})$. Suppose 
\begin{equation} \label{3-alpha} \ds x_1^{\l^{bc}_1+n-1}\bar{x}^{\tilde{\a}}=\end{equation}$$ x_1^{2k(n-1)-\l_n+n-1}x_2^{2k+1-\a_{n-1}}x_3^{2k+1-\a_{n-2}}\cdots x_{n-1}^{2k+1-\a_2}x_n^{2k+1-\a_1},$$
 appears in $P_2$ with coefficient $q_{\a}$ and
\begin{equation}\frac{x^{\l^{bc}+\d_n}}{x_1^{\l^{bc}_1+n-1}\bar{x}^{\tilde{\a}}}=:\bar{x}^{\eta+\d_{n-1}}	\end{equation}
 appears	in	$\bar{a}_{\d_{n-1}}^{2k+1}$ with coefficient $t_{\a}$. 	Here, $\eta=(\eta_1, \eta_2, \ldots, \eta_{n-1})$	 is a weak composition of $k(n-1)(n-2)$, with \begin{equation}\eta_i=\l^{bc}_{i+1}-\tilde{\a}_i=2k(n-1)-\l_{n-i}-(2k+1)+\a_{n-i}.
 \end{equation}
It is easily verified that $\n^{bc} = \eta$ and thus, by the inductive hypothesis, $s_{\a} = t_{\a}$. 

Now we compare the coefficients $p_{\a}$ and $q_{\a}$.  Since $ \ds \l_n=(2k+1)(n-1) - \sum_{i=1}^{n-1}\a_i$, we write  $p_{\a} \cdot  x^{\a}x_n^{\l_n}$ as \begin{equation}\label{m} \b_1x_1^{\a_1}x_n^{2k+1-\a_1}\cdot \b_2x_2^{\a_2}x_n^{2k+1-\a_2}\cdots \b_{n-1}x_{n-1}^{\a_{n-1}}x_n^{2k+1-\a_{n-1}}.\end{equation}
 Similarly, we write $q_{\a} \cdot   \ds x_1^{\l^{bc}_1+n-1}\bar{x}^{\tilde{\a}}$ as \begin{equation}\label{q}  \g_1 x_1^{\a_{n-1}}x_2^{2k+1-\a_{n-1}}\cdot \g_2x_1^{\a_{n-2}}x_3^{2k+1-\a_{n-2}}\cdots \g_{n-1}x_1^{\a_1} x_n^{2k+1-\a_1}.\end{equation}

 Each term $\b_ix_i^{\a_i}x_n^{2k+1-\a_i}$ in (\ref{m}) occurs only in the expansion of the term $(x_i-x_n)^{2k+1}$ in $P_1$. Similarly, each term $\g_ix_1^{\a_n-i}x_{i+1}^{2k+1-\a_{n-i}}$ in (\ref{q}) occurs only in the expansion of the term $(x_1-x_{i+1})^{2k+1}$ in $P_2$. Therefore, $\b_i=\g_{n-i}$ for all $i=1,2, \ldots, n-1$ and, consequently, $p_{\a}=q_{\a}$. \medskip
 
Summing up, we wrote \begin{equation}\langle a_{\d_n}^{2k},s_{\l} \rangle x^{\l+\d_n} = \sum_{\a} s_{\a}x^{\n+\d_{n-1}}\cdot p_{\a}x^{\a}x_n^{\l_n}\end{equation} and \begin{equation} \langle a_{\d_n}^{2k},s_{\l^{bc}} \rangle x^{\l^{bc}+\d_n}= \sum_{\a} t_{\a}x^{\eta+\d_{n-1}}\cdot q_{\a}x_1^{\l^{bc}_1+n-1}\bar{x}^{\tilde{\a}}.\end{equation}  
Since $p_{\a}=q_{\a}$ and $s_{\a}=t_{\a}$ for each $\a=(\a_1,\a_2, \ldots \a_{n-1})$ with $0\leq \a_i\leq \min(2k+1,\l_i+n-i)$, $i=1,2, \ldots, n-1$ and $\ds |\a|=(2k+1)(n-1)-\l_n$,   it follows that $\langle a_{\d_n}^{2k},s_{\l}\rangle =\langle a_{\d_n}^{2k},s_{\l^{bc}}\rangle$.

\end{proof}

\section{Simple recursive formulas}
\noindent The goal of this section is to establish some preliminary recursive formulas for $\langle a_{\d_{n+1}}^{2k},s_{\m}\rangle$  in 
terms of $\langle a_{\d_n}^{2k},s_{\l}\rangle$  when the diagram of $\m$ is obtained from the diagram of $\l$ by adding a certain 
configuration of boxes, called a \emph{tetris type shape}, to the top or to the left. For $k=1$ these results have been mentioned in \cite{STW}. 
\begin{theorem} \label{row} If $\l=(\l_1, \l_2, \ldots, \l_n)$ is a partition of $kn(n-1)$ with $\ell(\l)\leq n$ and $\m$ is the partition of $kn(n+1)$ given by $\m = (2kn,\l_1, \l_2, \ldots, \l_n)$, then
\begin{equation}\langle a_{\d_{n+1}}^{2k}, s_{\m}\rangle =\langle a_{\d_{n}}^{2k}, s_{\l}\rangle. \end{equation}\end{theorem}
Thus, adding the tetris type shape

 \begin{center}
 \raisebox{-.5cm}{\large{$\stackrel{2kn}{\overbrace{\hspace{3.4cm}}}$}}\\  
 {\small \young(\hfil\hfil\hfil{\ $\cdots$ \ }\hfil\hfil)}\\ \  \end{center}
   to the top of the diagram for $\l$ does not change the coefficient. If $k=1$, we denote this tetris type shape by $T_0(n,0)$. 

\begin{proof} The proof follows by induction from
\begin{equation}a_{\d_{n+1}}^{2k+1}=\prod_{i=2}^{n+1}(x_1-x_i)^{2k+1}\cdot \bar{a}_{\d_n}^{2k+1}.\end{equation} \end{proof}
\noindent {\bf Remark:} The theorem is also true if $\l$ is just a weak composition of $kn(n-1)$ with no more than one part equal to 0.\bigskip

\noindent {\bf Note:}	The result of the theorem for $k = 1$ is also noted in ($23$a) of \cite{STW} and in \cite[Section 6]{D}. 

\begin{cor} \label{unif}  If $\l = (2k(n-1),2k(n-2),\ldots,4k,2k,0) = 2k\d_n$, then
$\langle a_{\d_n}^{2k},s_{\l}\rangle=1$.\end{cor}
For the physical interpretation, when $k = 1$, the partition $\l$ in the corollary corresponds to the most evenly distributed of the Slater states (every third single particle angular momentum is filled) \cite{D}.\medskip

Using Theorem \ref{row} and Lemma \ref{bclemma}, we obtain the following corollary.
\begin{cor}\label{leftblock}  If $\l=(\l_1, \l_2, \ldots, \l_n)$ is a partition of $kn(n-1)$ with $\ell(\l)\leq n$ and $\m$ is the partition of $kn(n+1)$ given by $\m =\l+\langle (2k)^n\rangle= (\l_1+2k, \l_2+2k, \ldots, \l_n+2k)$, then
\begin{equation}\langle a_{\d_{n+1}}^{2k}, s_{\m}\rangle =\langle a_{\d_{n}}^{2k}, s_{\l}\rangle. \end{equation} \end{cor}
\noindent Thus adding the tetris type shape
\bigskip

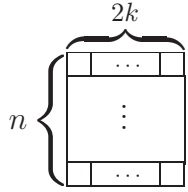
\begin{center}\begin{picture}(200,0)\put(100,0){\large{$\stackrel{2k}{\overbrace{\hspace{1.55cm}}}$}}\put(100,-10){\tiny \young(\hfil{ \ $\cdots$ \ \  }\hfil)} \put(78,-30){$n$}\put(86,-32){\LARGE $\biggl\{\stackrel{\ }{\ } \biggr.$} \put(100,-43.5){\line(0,0){33.5}} \put(120,-30){\vdots}\put(144.5,-43.5){\line(0,0){33.5}}\put(100,-51.5){\tiny \young(\hfil{ \ $\cdots$ \ \  }\hfil)}

\end{picture}\end{center}\vspace{.7in}

\noindent to the left of the diagram of $\l$ does not change the coefficient. If $k=1$, we denote this tetris type shape by $L(n)$.  \bigskip

\noindent {\bf Note:}	For $k = 1$ this is ($23$b) of \cite{STW}.\bigskip

\section{Recursive formulas in the case $k=1$}

For the remainder of the article we set $k = 1$. In this section we prove two non-trivial recursive formulas  involving tetris type shapes and present a third, conjectural, such formula. 

The following lemma and its corollary justify the assumption of the next theorem. 

\begin{lemma}
Suppose $\l \vdash n(n-1)$ with $n-1\leq  \ell(\l)\leq n$ and $\langle a_{\d_n}^{2}, s_{\l}\rangle\neq 0$. If $\l_n=\l_{n-1}=\ldots =\l_{n-i}=s$, then $i\leq s$, \emph{i.e.}, the maximum number of rows of size $s$ at the bottom of the diagram is $s+1$.  
\end{lemma}
\begin{proof} Suppose $\l_n =\l_{n-1}=\ldots= \l_{n-i} = s$ and write $$\ds x^{\l +\d_n}=(x_1^{\l_1+n-1}\cdots x_{n-i-1}^{\l_{n-i-1}+i+1})\cdot M_i,$$
where the monomial 

\begin{equation}\label{Mi}\ds M_i=x_{n-i}^{s+i} x_{n-i+1}^{s+i-1}\cdots x_{n-1}^{s+1}x_{n}^{s}\end{equation}
has degree
$$\frac{(s+i)(s+i+1)-(s-1)s}{2}=\frac{i^2+(2s+1)i+2s}{2}. $$ 

On the other hand, the product \begin{equation} \label{prod} \prod_{j=n-i+1}^n(x_{n-i}-x_j)^3\cdot \prod_{j=n-i+2}^n(x_{n-i+1}-x_j)^3\cdots (x_{n-1}-x_n)^3\end{equation} 
 contributes	powers	of	$x_{n-i}, x_{n-i+1}, \ldots, x_n$ 	to all	monomials in $a_{\d_n}^3$	and	thus
to $M_i$. However, each monomial in the product (\ref{prod}) has degree $\ds \frac{3i(i + 1)}{2}$. Comparing the degree of (\ref{prod}) and (\ref{Mi}), we see that $i \leq s$.	\end{proof}

\noindent {\bf Note:} We stated and proved the lemma for $k=1$ since only this case is needed in the article. However, the lemma is true for general $k$ which can be seen by replacing $3$ by $2k+1$ in (\ref{prod}). Then, for $\l \vdash kn(n-1)$ as in the lemma, comparing the degree of (\ref{prod}) and (\ref{Mi}), we have $ki \leq s$. Then, the maximum number of rows of size $s$ at the bottom of the diagram for $\l$ is $\ds \lfloor s/k\rfloor+1$. \medskip

We reformulate the previous lemma in terms of the box-complement of the partition $\l$ (set $m = s +1$).
\begin{cor} Suppose $\l \vdash n(n-1)$ with $n-1\leq  \ell(\l)\leq n$ and $\langle a_{\d_n}^{2}, s_{\l}\rangle\neq 0$. If $\l_1=\l_2=\ldots = \l_i=2n-m-1$, then $i \leq m$. 

\end{cor}\medskip

The first recursion formula of this section follows from the following theorem. 

\begin{theorem} \label{pretetris0}Let $1 \leq m \leq n$. Let $\l\vdash n(n-1)$ with $n-1\leq  \ell(\l)\leq n$ and $\l_1=\l_2=\ldots = \l_m=2n-m-1$. Then, \begin{equation}
\langle a_{\d_n}^{2}, s_{\l}\rangle=\langle a_{\d_m}^2,s_{\langle(m-1)^m\rangle}\rangle\cdot\langle a_{\d_{n-m}}^2,s_{( \l_{m+1},\l_{m+2},\ldots,\l_n)}\rangle.\end{equation}

\end{theorem}

\begin{proof} We have $$x^{\l+\d_n}=(x_1^{3n-m-2}x_2^{3n-m-3}\cdots x_m^{3n-2m-1})\cdot(x_{m+1}^{\l_{m+1}+n-m-1}\cdots x_{n-1}^{\l_{n-1}+1}x_n^{\l_n}).$$
We write $a_{\d_n}^3=a_{\d_m}^3\cdot B_m\cdot C_m,$ where $$B_m=\prod_{1\leq i \leq m<j\leq n}(x_i-x_j)^3  \ \mbox{and}\   C_m=\frac{a_{\d_n}^3}{a_{\d_m}^3B_m}=\prod_{m+1\leq i<j\leq n}(x_i-x_j)^3.$$ Note that $C_m$ is obtained from $a_{\d_{n-m}}^3$ via the substitution $$x_i \rightarrow x_{i+m}, \ x_j\rightarrow x_{j+m}.$$

\noindent  Since monomials in $B_m$ contain each $x_i$, $i=1, \ldots, m$, with exponent at most $3n-3m$,  the monomials in $a_{\d_m}^3$  contributing to $\langle a_{\d_n}^{2}, s_{\l}\rangle$ are of the form $$x_1^{2m-2}x_2^{2m-3} \cdots x_m^{m-1}\cdot E=x^{\langle(m-1)^m\rangle+\d_m}\cdot E,$$ where $E$ is a monomial in the variables $x_1, \ldots, x_m$. Since $$\ds \deg(a_{\d_m}^3)=\deg(x^{\langle(m-1)^m\rangle+\d_m})= \frac{3m(m-1)}{2},$$ we have $E=1$. Hence, the only monomial in $B_m$ contributing to $\langle a_{\d_n}^{2}, s_{\l}\rangle$ is $x_1^{3n-3m}x_2^{3n-3m}\cdots x_m^{3m-3n}$ (with coefficient $1$). \medskip

Therefore, $\langle a_{\d_n}^{2}, s_{\l}\rangle=\a_m \cdot \b_m$, where $$\a_m= \langle a_{\d_m}^2,s_{\langle(m-1)^m\rangle}\rangle$$ and $\b_m$ is the coefficient of $\ds x_{m+1}^{\l_{m+1}+n-m-1}\cdots x_{n-1}^{\l_{n-1}+1}x_n^{\l_n}$ in $C_m$, \emph{i.e.}, $$\b_m=\langle a_{\d_{n-m}}^2,s_{( \l_{m+1},\l_{m+2},\ldots,\l_n)}\rangle.$$ \end{proof}

\begin{cor} \label{tetris0}
Let $1 \leq m \leq n$. Let $\l\vdash n(n-1)$ with $n-1\leq  \ell(\l)\leq n$ and $\l_1=\l_2=\ldots = \l_m=2n-m-1$. Let $\m\vdash n(n+1)$ with parts $\m_1=\m_2=\ldots = \m_{m+1}=2n-m$ and  (if $m<n$) $\m_j=\l_{j-1}$ for $j=m+2, \ldots, n+1$. Then \begin{equation}\langle a_{\d_{n+1}}^{2}, s_{\m}\rangle= (-1)^m(2m+1)\langle a_{\d_n}^{2}, s_{\l}\rangle.\end{equation}
 
	\end{cor}
Thus, adding the tetris type shape

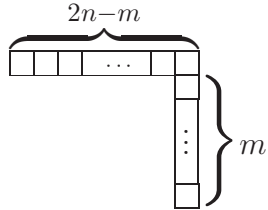
\begin{center}\begin{picture}(200,10) \put(100,-10){\large{$\stackrel{2n-m}{\overbrace{\hspace{2.5cm}}}$}}\ 
 \put(100,-20){\tiny  \young(\hfil\hfil\hfil{\ $\cdots$\ \  }\hfil\hfil)} \put(162.2,-29){\tiny  \young(\hfil)}\put(162.5,-61){\line(0,0){32}}  \put(165,-50){\vdots}\put(171.2,-61){\line(0,0){32}} \put(165,-50){{\LARGE $\biggl.\stackrel{\ }{\ } \biggr\}$}$m$}\put(162.2,-70){\tiny  \young(\hfil)}
 \end{picture}\end{center}\vspace{1in}
 
\noindent to the top of the diagram of $\l$ changes the coefficient by a multiple of $(-1)^m(2m + 1)$. We denote this tetris type shape by $T_0(n,m)$. \bigskip

For the physical interpretation, the partition $\l$ corresponds to the Slater state in which the angular momentum levels of the first $m$ particles are most closely bunched \cite{D}.

\begin{proof} 

As in the proof of Theorem \ref{pretetris0}, we have $\ds \langle a_{\d_{n+1}}^{2}, s_{\m}\rangle= \a_{m+1}\cdot \b_m$, where $\a_{m+1}= \langle a_{\d_{m+1}}^2,s_{\langle(m)^{m+1}\rangle}\rangle$. \medskip

\noindent By exercise $7.37$(b) of \cite{S}, \begin{equation}\label{alpham}\ds \a_m=\langle a_{\d_{m}}^2,s_{\langle(m-1)^m\rangle}\rangle=(-1)^{\binom{m}{2}}\cdot 1\cdot 3 \cdots (2m-1).\end{equation} Thus, from Theorem \ref{pretetris0}, it follows that  $$\ds \langle a_{\d_{n+1}}^{2}, s_{\m}\rangle= (-1)^m(2m+1)\ds \langle a_{\d_n}^{2}, s_{\l}\rangle.$$


\end{proof}

\note Corollary \ref{tetris0} implies the case $k=1$ of Theorem \ref{row}. \medskip

Now exercise $7.37$(c) of \cite{S} follows  easily by repeated use of  Corollary \ref{tetris0}. 

\begin{cor} \label{7.37.c} If $\l=\langle(n+i-1)^{n-i},(i-1)^i\rangle$, $1 \leq i\leq n$, then $$\langle a_{\d_n},s_{\l}\rangle=(-1)^{\frac{1}{2}(n-1)(n-2i)}[1\cdot 3 \cdots (2i-1)]\cdot[1\cdot 3 \cdots (2(n-i)-1)].$$

\end{cor}

We state conjecturally a similar combinatorial recursive property. The conjecture has been verified for $n \leq 6$ using  Maple.

\begin{con}\label{conjecture} Let $1 \leq m \leq n$. Let $\l\vdash n(n-1)$ with $n-1\leq  \ell(\l)\leq n$ and parts $\l_1=\l_2=\ldots = \l_m=2n-m-2$. Let $\m\vdash n(n+1)$ with parts $\m_1=2n-m$, $\m_2=\ldots = \m_{m+1}=2n-m-1$ and  (if $m<n$) $\m_j=\l_{j-1}$ for $j=m+2, \ldots, n+1$. Then \begin{equation}\langle a_{\d_{n+1}}^{2}, s_{\m}\rangle= (-1)^m(m+1)\langle a_{\d_n}^{2}, s_{\l}\rangle.\end{equation}

\end{con}

Thus adding the tetris type shape

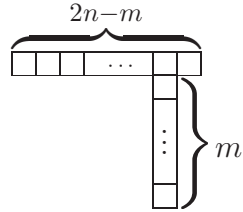
\begin{center}\begin{picture}(200,10) \put(100,-10){\large{$\stackrel{2n-m}{\overbrace{\hspace{2.5cm}}}$}}\ 
 \put(100,-20){\tiny  \young(\hfil\hfil\hfil{\ $\cdots$\ \  }\hfil\hfil)} \put(153.2,-29){\tiny  \young(\hfil)}\put(153.3,-61){\line(0,0){32}}  \put(156,-50){\vdots}\put(162.4,-61){\line(0,0){32}} \put(155,-51){{\LARGE $\biggl.\stackrel{\ }{\ } \biggr\}$}$m$}\put(153.1,-70){\tiny  \young(\hfil)}
 \end{picture}\end{center}\vspace{1in}

\noindent to the top of the diagram of $\l$ changes the coefficient by a multiple of $(-1)^m(m + 1)$. We denote this tetris type shape by $T_1(n,m)$. \bigskip

We can attempt to prove the conjecture in a manner similar to the proof of Theorem \ref{pretetris0}.

We have $x^{\l+\d_n}$ is equal to $$(x_1^{3n-m-3}x_2^{3n-m-4}x_3^{3n-m-5}\ldots x_m^{3n-2m-2})\cdot(x_{m+1}^{\l_{m+1}+n-m-1}\ldots x_{n-1}^{\l_{n-1}+1}x_n^{\l_n})$$and  $x^{\m+\d_{n+1}}$ is equal to $$(x_1^{3n-m}x_2^{3n-m-2}x_3^{3n-m-3}\ldots x_{m+1}^{3n-2m-1})\cdot(x_{m+2}^{\l_{m+1}+n-m-1}\ldots x_n^{\l_{n-1}+1}x_{n+1}^{\l_n}). $$\vspace{.01in}

We write $$ a_{\d_n}^3=a_{\d_m}^3 \cdot B_m \cdot C_m$$ with $B_m$ and $C_m$ as in the proof of Theorem \ref{pretetris0}. Similarly, we write  $$ a_{\d_{n+1}}^3=a_{\d_{m+1}}^3 \cdot \bar{B}_m \cdot \bar{C}_m,$$ where $$\bar{B}_m=\prod_{1\leq i \leq m+1<j\leq n+1}(x_i-x_i)^3$$ and $$  \bar{C}_m=\frac{a_{\d_{n+1}}^3}{a_{\d_{m+1}}^3\bar{B}_m}=\prod_{m+2\leq i<j\leq n+1}(x_i-x_i)^3.$$ 

\noindent By the argument in the proof of Theorem \ref{pretetris0},
the monomials in $a^3_{\d_m}$	contributing to $\langle a_{\d_n}^3,s_{\l} \rangle$ are of the form $$x_1^{2m-3}x_2^{2m-4}\cdots x_m^{m-2}\cdot F,$$
where $F$ is a  monomial of degree $m$ in the variables $x_1, x_2, \ldots, x_m$ of degree $\deg( a_{\d_m}^3)-\deg(x_1^{2m-3}x_2^{2m-4}\cdots x_m^{m-2})  =m$.\medskip

Similarly, monomials in $a^3_{\d{m+1}}$ contributing to 	 $\langle a_{\d_{n+1}}^3,s_{\m} \rangle$, are of the form $$x_1^{2m}x_2^{2m-2}x_3^{2m-3}\cdots x_{m+1}^{m-1}\cdot G,$$
 where $G$ is a  monomial of degree $m$ in  $x_1, x_2, \ldots, x_m, x_{m+1}$.\medskip
 
Let $l = (l_1,l_2,\ldots,l_m)$ be a partition of $m$ and set $l^* = (l_1,l_2,\ldots ,l_m,$ $ l_{m+1} = 0)$. Let $\a \in S_m$ be a permutation of  $\{1,2,
\ldots ,m\}$ and let $\b\in S_{m+1}$ be a permutation of $\{1, 2, \ldots, m, m + 1\}$. Denote by $\mathcal{C}(l, \a)$ the coefficient of
\begin{equation}x_1^{2m-3+l_{\a(1)}}x_2^{2m-4+l_{\a(2)}}\cdots x_m^{m-2+l_{\a(m)}}=x^{\langle(m-2)^m\rangle +\a(l)+\d_m}\end{equation}
in $a^3_{\d_m}$	and denote by $\overline{\mathcal{C}}(l,\b)$  the coefficient of $$x_1^{2m+l_{\b(1)}}x_2^{2m-2+l_{\b(2)}}x_3^{2m-3+l_{\b(3)}}\cdots x_m^{m+l_{\b(m)}}x_{m+1}^{m-1+l_{\b(m+1)}}=$$\begin{equation}x^{\langle m, (m-1)^m\rangle +\b(l^*)+\d_{m+1}}\end{equation}
 in $a^3_{\d_{m+1}}$. Denote by $\mathcal{D}($l) the coefficient of $$x_1^{3n-3m-l_{\a(1)}}x_2^{3n-3m-l_{\a(2)}}\cdots x_m^{3n-3m-l_{\a(m)}}x_{m+1}^{\l_{m+1}+n-m-1}\cdots x_n^{\l_n}$$
in $B_m \cdot C_m$. Since $B_m \cdot C_m$ is symmetric in $x_1,x_2,\ldots ,x_m$, $\mathcal{D}(l)$ does not depend on the permutation $\a$. The coefficient of $$x_1^{3n-3m-l_{\b(1)}}x_2^{3n-3m-l_{\b(2)}}\cdots x_m^{3n-3m-l_{\b(m)}}x_{m+1}^{3n-3m-l_{\b(m+1)}}x_{m+2}^{\l_{m+1}+n-m-1}\cdots x_{n+1}^{\l_n}$$
in $\bar{B}_m \cdot  \bar{C}_m$ is again $\mathcal{D}(l)$ (since at least one of $l_{\b(j)}$ equals $0$). Then $$\langle a_{\d_n}^2,s_{\l}\rangle = \sum_{l}\left(\mathcal{D}(l)\sum_{\a}\mathcal{C}(l,\a)\right)$$ and $$\langle a_{\d_{n+1}}^2,s_{\m}\rangle = \sum_{l}\left(\mathcal{D}(l)\sum_{\b}\overline{\mathcal{C}}(l,\b)\right),$$ 
where the first summation is, in each case, over all partitions $l$ of $m$ and the second summation is over all \emph{distinct} permutations $\a$ of the parts of $l = (l_1,l_2,\ldots,l_m)$, respectively all \emph{distinct} permutations $\b$ of the parts  $\{l_1,l_2,\ldots ,l_m, l_{m+1} = 0\}$.

To prove the conjecture, it remains to show that for each partition $l = (l_1,l_2,\ldots,l_m)$ of $m$ of length at most $m$,\begin{equation} \label{conjperm}\sum_{\b }\overline{\mathcal{C}}(l,\b)=(-1)^m(m+1)\sum_{\a}\mathcal{C}(l,\a).\end{equation}

If $m = 1$, the conjectural relation (\ref{conjperm}) can be verified directly. We have $l = (1)$ and $l^* = (1, 0)$. The left hand side adds the respective coefficients of $x^3_1$ and of $x^2_1x_2$ in $\a^3_{\d_2}$ and it equals $-2$. The sum on the right hand side has only one element, the coefficient of $1$ in $a^3_{\d_1}$, which is $1$. Therefore, the right hand side also equals $-2$. This proves case $m = 1$ of Conjecture \ref{conjecture}.
\begin{prop}\label{conjcase1}Let $\l=(2n-3,\l_2, \ldots, \l_n) \vdash n(n-1)$ and $\m=(2n-1, 2n-2, \l_2, \ldots, \l_n)\vdash n(n+1)$.  Then \begin{equation}\langle a_{\d_{n+1}}^2, a_{\m}\rangle =-2 \langle a_{\d_n}^2, s_{\l} \rangle. \end{equation}\end{prop}

Next, we prove the conjecture for $m = n - 1$. This will be needed for the proof of the last recursive formula of the article. We first introduce some definitions following \cite[Chapter 7]{S}. Denote by $f_{\l}$ the number of standard Young tableaux (SYT) of shape $\l$. Given a Young diagram $\l$ and a square $u = (i,j)$ of $\l$, 
define the content, $c(u)$, of $\l$ at $u = (i, j)$ by
$$c(u) = j - i.$$ 
If $\l$ is a partition of $n(n - 1)$ that can be written as $\l=\eta+\langle (n-2)^n\rangle$, where $\eta$  is a partition of $n$, then, by \cite[Exercise 7.37.d]{S} \begin{equation}\label{7.37.d}\langle a_{\d_n}^{2}, s_{\l}\rangle=(-1)^{\binom{n}{2}}f_{\eta}\prod_{s\in \eta}(1-2c(s)).\end{equation} As noted in \cite{H}, $$f_{\l}=\sum_{\n\in \l\setminus 1} f_{\n},$$ where $\l\setminus 1$ is the set of partitions obtained from $\l$ by removing a corner. (This formula follows directly from the construction of standard Young tableaux.) \bigskip

Consider the partitions $\l=\langle (n-1)^n\rangle \vdash n(n-1)$ and $\m=\langle n+1, n^{n-1}, n-1) \vdash n(n+1)$. We have  \begin{equation} \l=\langle(n-1)^n\rangle=\langle 1^n\rangle+\langle (n-2)^n\rangle \end{equation}
and \begin{equation} \m=\langle2,1^{n-1}\rangle+\langle(n-1)^{n+1}\rangle.\end{equation}\medskip

\noindent Then, by (\ref{7.37.d}) and the immediate fact that $f_{\langle 1^n\rangle}=1$ and $f_{\langle2,1^{n-1}\rangle}=n$, it follows that 
$\langle a_{\d_n}^{2}, s_{\l}\rangle$ equals \begin{equation} \label{conjl}(-1)^{\binom{n}{2}}f_{\langle 1^n\rangle}\prod_{s\in \langle 1^n\rangle}(1-2c(s))=(-1)^{\binom{n}{2}}1\cdot 3\cdot 5 \cdots (2n-1)\end{equation} and $\langle a_{\d_{n+1}}^{2}, s_{\m}\rangle$ equals $$ (-1)^{\binom{n+1}{2}}f_{\langle2,1^{n-1}\rangle}\prod_{s\in \langle2,1^{n-1}\rangle}(1-2c(s))=$$ \begin{equation}\label{conjm}(-1)^{\binom{n+1}{2}}n(-1)\cdot 1\cdot 3\cdot 5 \cdots (2n-1). \end{equation}
Comparing (\ref{conjl}) and (\ref{conjm}), proves Conjecture \ref{conjecture} in the case $m=n-1$.

\begin{prop} \label{conjn-1} Let $\l=\langle (n-1)^n\rangle \vdash n(n-1)$ and $\m=\langle n+1, n^{n-1}, n-1) \vdash n(n+1)$. Then
\begin{equation}\langle a_{\d_{n+1}}^2, s_{\m} \rangle =(-1)^{n-1} n \langle a_{\d_n}^2, s_{\l} \rangle.	\end{equation} \end{prop}

Before considering the last recursive formula, we prove another  helpful lemma.

First, some notation. Suppose $\n$ is a partition of $n-1$ and $\m$ is a partition of $n$ containing $\n$. Then the shape $\m$ is obtained by adding a square to the shape $\n$. We denote by $c(\m/\n)$ the content of the square $\m/\n$ (i.e., the square added to the shape $\n$ in order to obtain the shape $\m$) in the shape $\m$.\medskip

Recall that $$f_{\n}=\sum_{\eta \in \n\setminus 1}f_{\eta},$$
where $\n\setminus 1$ is the set of shapes obtained from $\n$ by removing one square. We write this fact as \begin{equation}\label{remove}f_{\n}=\sum_{\stackrel{\eta\vdash n-2}{\eta \subseteq \n}}f_{\eta}.\end{equation}

\begin{lemma} \label{conj2-lemma} Let $\n$ be a partition of $n-1$.  We have $$nf_{\n}=\sum_{\stackrel{\m \vdash n}{\n\subseteq \m}}f_{\m}(1-2c(\m/\n)).$$  \end{lemma}

\begin{proof} We prove the lemma by induction on $n$. If $n=2$, the statement of the lemma is true by inspection. (Actually, if $n=1$ the lemma is also true, assuming $f_{\emptyset}=1$.)\bigskip

Assume the statement is true for all partitions  of $n-1$. Now let $\n$ be a partition of $n$.  We need to show that \begin{equation}\label{statement}(n+1)f_{\n}=\sum_{\stackrel{\m \vdash n+1}{\n\subseteq \m}}f_{\m}(1-2c(\m/\n)).\end{equation}

Consider first the left hand side of (\ref{statement}). Using (\ref{remove}),  we have \begin{equation} \label{removelhs} (n+1)f_{\n}=f_{\n}+nf_{\n}=f_n+\sum_{\stackrel{\eta \vdash n-1}{\eta\subseteq \n}}nf_{\eta}.\end{equation} By the inductive hypothesis,  \begin{equation}\label{addlhs}(n+1)f_{\n}=f_{\n}+\sum_{\stackrel{\eta \vdash n-1}{\eta\subseteq \n}}\left(\sum_{\stackrel{\m \vdash n}{\eta\subseteq \m}} f_{\m}(1-2c(\m/\eta))\right).\end{equation}

Note that   in (\ref{removelhs}) we remove a square from the shape $\n$ whenever possible and in (\ref{addlhs}) we add a square to the obtained shape whenever possible. There are two possibilities: \bigskip

\noindent (i) The added square is precisely the removed square. Then, $f_{\m}=f_{\n}$. 

\noindent (ii) The added square is different from the removed square. In this case, the operations of removing and adding squares commute. \bigskip

We separate these possibilities in the sum above. Thus, $$(n+1)f_{\n}=f_{\n}+f_{\n}\sum_{\stackrel{i=1}{\n_i>\n_{i+1}}}^{\ell(\n)} (1-2(\n_i-i))+
\sum_{\stackrel{\eta \vdash n-1}{\eta\subseteq \n}}\left(\sum_{\stackrel{\m \vdash n}{\stackrel{\eta\subseteq \m}{\mu \neq \n}}} f_{\m}(1-2c(\m/\eta))\right).$$ Using the commutativity of the operations of removal and addition of a square in case (ii) above, we have \begin{equation} \label{lhs}(n+1)f_{\n}=f_{\n}+f_{\n}\sum_{\stackrel{i=1}{\n_i>\n_{i+1}}}^{\ell(\n)} (1-2(\n_i-i))+
\sum_{\stackrel{\m \vdash n+1}{\n\subseteq \m}}\left(\sum_{\stackrel{\eta \vdash n}{\stackrel{\eta\subseteq \m}{\eta \neq \n}}} f_{\eta}(1-2c(\m/\n))\right).\end{equation} \medskip

Now we consider the right hand side of (\ref{statement}). 

Using (\ref{remove}), we have
$$\sum_{\stackrel{\m \vdash n+1}{\n\subseteq \m}}f_{\m}(1-2c(\m/\n))= \sum_{\stackrel{\m \vdash n+1}{\n\subseteq \m}}\left((1-2c(\m/\n))\sum_{\stackrel{\eta\vdash n}{\eta \subseteq \m}}f_{\eta}  \right).$$   Separating the possibilities (i) and (ii), the right hand side equals
\begin{equation}\label{rhs}f_{\n}(1-2\n_1) +f_{\n}\sum_{\stackrel{i=1}{\n_i>\n_{i+1}}}^{\ell(\n)}(1-2(\n_{i+1}-i))+\sum_{\stackrel{\m \vdash n+1}{\n\subseteq \m}}\left((1-2c(\m/\n))\sum_{\stackrel{\eta\vdash n}{\stackrel{\eta \subseteq \m}{\eta \neq \n}}}f_{\eta}  \right).\end{equation}\bigskip

To show that  (\ref{lhs}) and (\ref{rhs}) are equal, we need to show that \begin{equation}\label{nof}1+\sum_{\stackrel{i=1}{\n_i>\n_{i+1}}}^{\ell(\n)} (2i+1-2\n_i))=1-2\n_1+\sum_{\stackrel{i=1}{\n_{i}>\n_{i+1}}}^{\ell(\n)}(2i-1-2\n_i),\end{equation} which is true since 

$$\sum_{\stackrel{i=1}{\n_{i-1}>\n_i}}^{\ell(\n)}(\n_i-\n_{i+1})=\n_1.$$ This concludes the proof of the lemma

\end{proof}

\noindent Lemma \ref{conj2-lemma} also follows form results in \cite{K}.

\begin{theorem} \label{lefttetris} Let $\l\vdash n(n-1)$ with $\ell(\l)=n-1$ and $\l_{n-1}\geq n-1$ and let $\m \vdash n(n+1)$ be given by $\m=(\l_1+1, \l_2+1,\ldots, \l_{n-1}+1,n,1)$. Then,
\begin{equation}\langle a_{\d_{n+1}}^{2}, s_{\m}\rangle= (-1)^n3n\langle a_{\d_n}^{2}, s_{\l}\rangle.\end{equation}	\end{theorem}

 Thus adding the tetris type shape
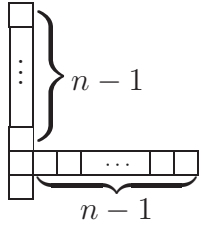
\vspace{.2in}

 \begin{center}
\begin{picture}(105,39)(200,0)\put(123.7,40){\tiny \young(:::::::::\hfil)} 
\put(205.5,2){\line(0,0){38}}  \put(208,18){\vdots}\put(214.2,2){\line(0,0){38}} \put(207,16){{\LARGE $\biggl.\stackrel{\ }{\ } \biggr\}$}$n-1$}
\put(123.8,-7){\tiny \young(:::::::::\hfil)}\put(205.2,-16){\tiny \young(\hfil\hfil\hfil{\ $\cdots$\ \  }\hfil\hfil)} \put(205.2,-25){\tiny \young(\hfil)} \put(216,-33){{$\stackrel{\underbrace{\hspace{2.1cm}}}{n-1}$}}\end{picture} \end{center} \vspace{.5in}

\noindent  to the left of the diagram of $\l$ changes the coefficient by a multiple of $(-1)^n3n$. We denote this tetris type shape by $L_1(n)$. 
\begin{proof} 

\noindent \emph{Case I:} $\l_{n-1}\geq n$. Then $\l=\langle n^{n-1}\rangle$ and $\m=\langle (n+1)^{n-1},n,1\rangle$. Using Corollary \ref{7.37.c} with $i=1$, we have  $$\langle a_{\d_n}^{2}, s_{\l}\rangle=(-1)^{\binom{n-1}{2}}1\cdot 3 \cdot 5 \cdots (2n-3). $$

\noindent We have $$ \m^{bc}=\langle 2n-1, n, (n-1)^{n-1}\rangle = (n,1)+\langle (n-1)^{n+1}\rangle.$$\medskip

\noindent By (\ref{7.37.d}), $$\langle a_{\d_{n+1}}^{2}, s_{\m^{bc}}\rangle=(-1)^{\binom{n+1}{2}}f_{(n,1)} \prod_{s\in (n,1)}(1-2c(s)).$$\medskip

\noindent Since $f_{(n,1)}=n$ and, by Lemma \ref{bclemma}, $\ds \langle a_{\d_{n+1}}^{2}, s_{\m}\rangle=\langle a_{\d_{n+1}}^{2}, s_{\m^{bc}}\rangle$, we have $$\langle a_{\d_{n+1}}^{2}, s_{\m}\rangle=(-1)^{\binom{n+1}{2}}n(-1)^{n-1}3\cdot1\cdot 3 \cdot 5 \cdots (2n-3)=(-1)^n3n\langle a_{\d_n}^{2}, s_{\l}\rangle.$$

\noindent \emph{Case II:} $\l_{n-1}=n-1$. Thus, $\l=\langle (n-1)^{n-1} \rangle +\n$, where $\n=(\n_1, \n_2, \ldots, \n_{n-1})\vdash n-1$. The last part of $\n$ can only be $0$ or $1$. If $\n_{n-1}=1$, then we are in \emph{Case I}. Therefore, we assume $\n_{n-1}=0$. 

Using Corollary \ref{leftblock}, we have $$\langle a_{\d_n}^2, s_{\l}\rangle = \langle a_{\d_n}^2, s_{\l/\langle 2^{n-1}\rangle}\rangle,$$ where $\l/\langle 2^{n-1}\rangle=(\l_1-2, \l_2-2, \ldots, \l_{n-1}-2)$.  
 Since $\l/\langle 2^{n-1}\rangle=\langle (n-3)^{n-1}\rangle +\n$ is a partition of $(n-1)(n-2)$, we can use (\ref{7.37.d}) to obtain \begin{equation}\label{lambdaf}\langle a_{\d_n}^2, s_{\l}\rangle=(-1)^{\binom{n-1}{2}}f_{\n}\prod_{s \in \n}(1-2c(s)).\end{equation} 

Now let us consider the partition $\m=(n+\n_1, n+\n_2, \ldots, n+\n_{n-1},n,1)$. We have $$x^{\m+\d_{n+1}}=x_1^{2n+\n_1}x_2^{2n+\n_2-1}\cdots x_i^{2n-\n_i-i+1}\cdots x_{n-2}^{n+3+\n_{n-2}}x_{n-1}^{n+2+\n_{n-1}}x_n^{n+1}x_{n+1}.$$ We write $a_{\d_{n+1}}^3$ as $$a_{\d_{n+1}}^3=a_{\d_{n}}^3\prod_{i=1}^{n}(x_i-x_{n+1})^3.$$
For each $i=1,2,\ldots, n-1$, the product $\ds \prod_{i=1}^{n}(x_i-x_{n+1})^3$ contributes $$-3x_i^2x_{n+1}x_1^3x_2^3\cdots x_{i-1}^3x_{i+1}^3\cdots x_n^3$$ and  $a_{\d_{n}}^3$ contributes $$c_i\cdot x_1^{2n+\n_1-3}x_2^{2n+\n_2-4}\cdots x_{i-1}^{2n+\n_{i-1}-i-1}x_i^{2n+\n_i-i-1} x_{i+1}^{2n+\n_{i+1}-i-3}\cdots x_{n-1}^{n+\n_{n-1}-1}x_n^{n-2}$$ to $x^{\m+\d_{n+1}}$. \medskip

\noindent Note that each monomial in $a_{\d_n}^3$ has the property that no two variables have the same exponent. Thus, if for some $i \geq 2$ we have $\n_{i-1}=\n_i$, then $c_i=0$. (There is no contribution when $i=n$ because $\n_{n-1}=0$.)\medskip

For  each $i=1,2, \ldots, \ell(\n)+1$, such that $\n_{i-1}>\n_i$ (by convention, $\n_0>\n_1$), we have $c_i=\langle a_{\d_n}^2, s_{\eta^{(i)}}\rangle$, where $\ds \eta^{(i)}=( \eta^{(i)}_1, \eta^{(i)}_2, \ldots , \eta^{(i)}_n)$ has parts $\ds \eta^{(i)}_j=n-2+\n_j$ if $j \not = i, n$, $\eta^{(i)}_i=n-1+\n_i$, and $\eta^{(i)}_n=n-2$. Thus 
 \begin{equation}\eta^{(i)}=\langle (n-2)^n\rangle +\tilde{\n}^{(i)},\end{equation} where $\tilde{\n}^{(i)}$ is the partition of $n$ obtained from $\n$ by adding  a box at the end of the $i$th row, \emph{i.e.}, \begin{equation}\tilde{\n}^{(i)}=(\n_1, \n_2, \ldots, \n_{i-1}, \n_i+1, \n_{i+1}, \ldots, \n_{n-1}).\end{equation} To find $c_i$ we  use (\ref{7.37.d}). We have \begin{equation}\label{nuif}c_i=(-1)^{\binom{n}{2}}f_{\tilde{\n}^{(i)}}\prod_{s \in \tilde{\n}^{(i)}}(1-2c(s)).\end{equation} 
 
 We have \begin{equation}\prod_{s \in\tilde{\n}^{(i)}}(1-2c(s))=(1-2c(i,\n_i+1))\prod_{s \in \n}(1-2c(s)).\end{equation}

Thus, using (\ref{lambdaf}) and (\ref{nuif}), in order to prove the theorem, we need to show that \begin{equation}nf_{\n}=\sum_{\stackrel{i=1}{\n_{i-1}>\n_i}}^nf_{\tilde{\n}^{(i)}}(1-2c(i,\n_i+1)).\end{equation} Note that the terms for $i=1$ and $i=\ell(\n)+1$ are always included in the sum.

This is precisely the statement of  Lemma \ref{conj2-lemma}. 

\end{proof}
\section {Applications}

In  \cite[Section 6]{D}, Dunne provides (without proof) closed formulas for several specific Slater states. They correspond to close formulas for $\langle a_{\d_n}^2, s_{\l} \rangle$ for specific (very symmetric) partitions $\l$. 
In this section, we use the recursive rules of the previous sections to prove some of these formulas. We also use our rules to explain recursive patterns observed by Dunne in the same section. We adapt the notation to match that of our previous sections and paraphrase Dunne's physical\medskip \ explanations. 

Dunne starts by mentioning that $ \langle a_{\d_n}^{2},s_{\l}\rangle=1$ for the most uniformly distributed of the Staler states, \emph{i.e.}, the state corresponding to $\l = (2(n-1),2(n-2),\ldots,4,2,0)$. This is the result of Corollary \ref{unif}. Next, he gives the coefficient for the situation in which the angular momentum levels are most closely bunched, \emph{i.e.}, $\l=\langle (n-1)^n\rangle$. This is our formula (\ref{conjl}): $ \langle a_{\d_n}^{2},s_{\l}\rangle=(-1)^{\binom{n}{2}}1\cdot 3\cdot 5 \cdots (2n-1)$. Note that in each of these two cases $\l=\l^{bc}$.\medskip

The next case, $\l=\langle (n)^{n-1}, 0\rangle$,  is not invariant under taking the box complement. Here one electron is in the $0$ angular momentum state and the remaining $n-1$ electrons are bunched together. The coefficient is $ \langle a_{\d_n}^{2},s_{\l}\rangle=(-1)^{\binom{n-1}{2}}1\cdot 3\cdot 5 \cdots (2n-3)$, which is the result of Corollary \ref{7.37.c} with $i=1$. \medskip

The above cases have all been noted previously in the combinatorics literature (the first case in \cite{STW} and the last two as exercises in \cite{S}, for example). We mention them here for completion and to show how they fit in the framework of the recursion formulas. The interesting applications of our rules come in the next batch of Dunne's closed formulas.  

Starting with the maximally bunched state $\langle (n-1)^n\rangle$ and successively moving the extreme inner and outer electrons in and out (respectively) by one step, the formulas given by Dunne correspond to:  

\begin{equation}\label{moveone} \langle a_{\d_n}^{2},s_{\langle n, (n-1)^{n-2}, n-2 \rangle }\rangle=(-1)^{\binom{n}{2}+1}(n-1) \cdot 1\cdot 3\cdot 5 \cdots (2n-3)\end{equation}
\begin{equation}\label{movetwo} \langle a_{\d_n}^{2},s_{\langle n+1, (n-1)^{n-2}, n-3 \rangle }\rangle=(-1)^{\binom{n}{2}+1}n(n-1) \cdot 1\cdot 3\cdot 5 \cdots (2n-5)\end{equation}
$$\vdots $$
\begin{equation}\label{moveall} \langle a_{\d_n}^{2},s_{\langle 2(n-1), (n-1)^{n-2}, 0 \rangle }\rangle=(-1)^{\binom{n}{2}+1} \cdot 1\cdot 3\cdot 5 \cdots (2n-5)\end{equation}\bigskip

To prove (\ref{moveone}), notice that $\langle n, (n-1)^{n-2}, n-2 \rangle$ is obtained from $\langle (n-2)^{n-1}\rangle$ by adding to its top a tetris type shape $T_1(n-1,n-2)$. By Proposition \ref{conjn-1} we have $$\langle n, (n-1)^{n-2}, n-2 \rangle= (-1)^{n-2}(n-1) \langle (n-2)^{n-1}\rangle$$ and thus, by (\ref{conjl}), $$\langle n, (n-1)^{n-2}, n-2 \rangle= (-1)^{\binom{n-1}{2}+n-2}(n-1)  \cdot 1\cdot 3\cdot 5 \cdots (2n-3),$$ which is equivalent to (\ref{moveone}). \bigskip

To prove (\ref{moveall}), we use Theorem \ref{row} to obtain $$ \langle a_{\d_n}^{2},s_{\langle 2(n-1), (n-1)^{n-2}, 0 \rangle }\rangle= \langle a_{\d_n}^{2},s_{\langle (n-1)^{n-2} \rangle }\rangle.$$ Then, by  Corollary \ref{7.37.c} with $i=1$, we have
$$ \langle a_{\d_n}^{2},s_{\langle 2(n-1), (n-1)^{n-2}, 0 \rangle }\rangle=(-1)^{\binom{n-2}{2}}\cdot 1\cdot 3\cdot 5 \cdots (2n-5), $$which is equivalent to (\ref{moveall}). \bigskip

We can also prove the formula that would naturally come before (\ref{moveall}), \emph{i.e.}, the coefficient of $s_{\langle 2n-3, (n-1)^{n-2}, 1 \rangle }$ in the decomposition of $a_{\d_n}^2$. Notice that $\langle 2n-3, (n-1)^{n-2}, 1 \rangle$ is obtained from $\langle 2n-4, (n-2)^{n-3} \rangle$ by adding to the left a tetris type shape $L_1(n-1)$. By  Theorem \ref{lefttetris},  $$ \langle a_{\d_n}^{2},s_{\langle 2n-3, (n-1)^{n-2}, 1\rangle}\rangle= (-1)^{n-1}3(n-1)\langle a_{\d_n}^{2},s_{\langle 2n-4, (n-2)^{n-3}\rangle} \rangle.$$

\noindent Since $\ds \langle a_{\d_n}^{2},s_{\langle 2n-4, (n-2)^{n-3}\rangle} \rangle= \langle a_{\d_n}^{2},s_{\langle (n-2)^{n-3}\rangle} \rangle$ by Theorem \ref{row}, we can use  Corollary \ref{7.37.c} with $i=1$ to obtain $$ \langle a_{\d_n}^{2},s_{\langle 2n-3, (n-1)^{n-2}, 1\rangle}\rangle=(-1)^{\binom{n-3}{2}+n-1}3(n-1)\cdot 1\cdot 3\cdot 5 \cdots (2n-7).$$ Thus, the formula preceding (\ref{moveall}) should be 
$$ \langle a_{\d_n}^{2},s_{\langle 2n-3, (n-1)^{n-2}, 1\rangle}\rangle=(-1)^{\binom{n}{2}+1}3(n-1)\cdot 1\cdot 3\cdot 5 \cdots (2n-7).$$

The recursions established in this article do not help prove (\ref{movetwo}) and the rest of the formulas alluded to above. On the other hand, the existence of these formulas is encouraging evidence that  further recursions must exist (perhaps in the form of adding/removing "broken" tetris type shapes). \medskip

Dunne's next suggestion is to start with the maximally distributed state, corresponding to  $\l=(2(n-1), 2(n-2), \ldots, 4, 2, 0)$, and make local shifts of electrons between angular momentum levels. In terms of partitions and Young diagrams, this corresponds to removing the last box in the $j$th row of $\l$ above and adding it to the the end of the $(j+1)$st  row. He notes "the remarkable fact that such an operation always changes the coefficient by a factor of $-3$." He generalizes the observation to the situation when the last box in the $j$th row of $\l$ is removed and added to the end of the $(j+k)$th  row. We prove this formula in the following proposition.

\begin{prop} Fix an integer $j$ with $1\leq j \leq n-1$ and let $k$ be an integer such that $j+1\leq l\leq n$. If $\n=(\n_1, \n_2, \ldots, \n_n) \vdash n(n-1)$ is given by $\n_i=2(n-i)$ if $i \neq j, l$, and $\n_j=2(n-j)-1$,  $\n_l=2(n-l)+1$, then
\begin{equation}\langle a_{\d_n}^2, s_{\n} \rangle=(-1)^{l-j}\cdot 3 \cdot 2^{l-j-1}.
\end{equation}\end{prop}

\begin{proof}

Start with the Young diagram for $\n$,  remove the top $j-1$ rows, \emph{i.e.}, the tetris type shapes $T_0(n-1,0), T_0(n-2, 0), \ldots T_0(n-j+1,0)$. By Theorem \ref{row}, we have $$ \langle a_{\d_n}^2, s_{\n}\rangle = \langle a_{\d_{n-j+1}}^2, s_{(\n_j, \n_{j-1}, \ldots, \n_n)}\rangle.$$ Next,  remove the $l-j-1$ tetris type shapes $T_1(n-j,1), T_1(n-j-1,1)$, $ \ldots, T_1(n-l+2,1)$.  By Proposition \ref{conjcase1}, we have $$\langle a_{\d_{n-j+1}}^2, s_{(\n_j, \n_{j-1}, \ldots, \n_n)}\rangle=(-2)^{l-j-1}\langle  a^2_{\d_{n-l+2}}, s_{(\n_{l-1}-1, \n_l, \ldots, \n_n)}\rangle.$$ Notice that $\n_{l-1}-1=\n_l$. Remove a tetris type shape  $T_0(n-l+1,1)$. By Corollary \ref{tetris0}, we obtain $$\langle  a^2_{\d_{n-l+2}}, s_{(\n_{l-1}-1, \n_l, \ldots, \n_n)}\rangle=-3\langle  a^2_{\d_{n-l+1}}, s_{(\n_{l}-1, \n_{l+1}, \ldots, \n_n)}\rangle.$$ If $l=n$, $(\n_{l}-1, \n_{l+1}, \ldots, \n_n)$ is the empty partition. Otherwise, it is $(2(n-l), 2(n-l-1), \ldots, 4,2,0)$. In either case $$\langle  a^2_{\d_{n-l+1}}, s_{(\n_{l}-1, \n_{l+1}, \ldots, \n_n)}\rangle=1.$$ Combining these results completes the proof. 

\end{proof}

Finally, we use the recursions of this article to explain some recursive properties observed by Dunne.  If we write $\#(n)$ for the number of Schur functions appearing in the decomposition of $a_{\d_n}^2$, he notes that, with a consistent ordering of the coefficients (as in the tables at the end of \cite{D}), \vspace{.1in}

(i) "the first $\#(n-1)$ coefficients for $n$ particles coincide with  all the coefficients for $n-1$ particles;\vspace{.1in}

(ii)   the next $\#(n-2)$ coefficients of the $n$ particle problem are given by $-3$ times the $\#(n-2)$ coefficients of the $n-2$ particle problem; \vspace{.1in}

(iii) the next $\#(n-3)$ coefficients of the $n$ particle problem are given by $6$ times the $\#(n-3)$ coefficients of the $n-3$ particle problem; \vspace{.1in}

(iv) the next $\#(n-4)$ coefficients of the $n$ particle problem are given by $-12$ times the $\#(n-4)$ coefficients of the $n-4$ particle problem, etc." \vspace{.1in}

This can be explained as follows. \vspace{.1in}

(i) Start with a partition $\l$  corresponding to a  Schur function appearing in the decomposition of $a^2_{\d_{n-1}}$ and add to its left a tetris type $L(n-1)$ to obtain a partition $\m$. Then, by Corollary \ref{leftblock}, $\langle a_{\d_n}^2, s_{\m}\rangle = \langle a_{\d_{n-1}}^2, s_{\l}\rangle$. (This correspondence matches Dunne's ordering in the tables at the end of his article.)\vspace{.1in}

(ii) Start with a partition $\l$ corresponding to a Schur functions appearing in the decomposition of $a^2_{\d_{n-2}}$ and add to its top a tetris type shape $T_0(n-2,0)$ (a row of length $2n-4$) to obtain a partition $\m$ whose Shur function appears in the decomposition of $a^2_{\d_{n-1}}$ with   coefficient  $\langle a_{\d_{n-2}}^2,s_{\l}\rangle$ (by Theorem \ref{row}). Then, add to the top of $\m$ a tetris type shape $T_0(n-1,1)$ to obtain a partition $\n$  whose Shur function appears in the decomposition of $a^2_{\d_{n}}$ with   coefficient  $-3\langle a_{\d_{n-2}}^2,s_{\l}\rangle$ (Corollary \ref{tetris0}).\vspace{.1in}

(iii) Start with a partition $\l$ corresponding to a Schur functions appearing in the decomposition of $a^2_{\d_{n-3}}$ and follow the steps in (ii), \emph{i.e.}, add a tetris type shape $T_0(n-3,0)$ to the top of $\l$ to obtain $\m$, and a tetris type shape $T_0(n-2,1)$ to the top of $\m$ to obtain $\n$. The Schur function for $\n$ appears in the decomposition of $a^2_{\d_{n-1}}$. Now add to the top of $\n$ a tetris type shape $T_1(n-1,1)$ to obtain a partition $\eta$. By Proposition \ref{conjcase1} and (ii), we have   $\langle a_{\d_n}^2, s_{\eta}\rangle =-2 \langle a_{\d_{n-1}}^2, s_{\n}\rangle=6  \langle a_{\d_{n-3}}^2, s_{\l}\rangle$. \vspace{.1in}

(iv) Start with a partition $\l$ corresponding to a Schur functions appearing in the decomposition of $a^2_{\d_{n-4}}$ and follow the steps in (iii). Thus, $\eta$, which is a partition for the $n-1$ particle problem, is obtained from $\l$ by adding to its top, in order, $T_0(n-4,0)$, $T_0(n-3,1)$ and $T_1(n-2,1)$. Add to the top of $\eta$  another tetris type $T_1(n-1,1)$ to obtain a partition $\xi$. Then, by Proposition \ref{conjcase1} and (iii), we have $\langle a_{\d_n}^2, s_{\xi}\rangle =-12 \langle a_{\d_{n-4}}^2, s_{\l}\rangle$. \vspace{.1in}


\section{Concluding remarks} 

The recursive formulas of this article together with the box-complement lemma give $15$ of the $16$ coefficients in the $n=4$ problem in terms of the coefficients for $n=3$ and $48$ of the $59$ coefficients in the $n=5$ problem in terms of the coefficients for $n=4$. This is a considerable improvement to the recursive observation in \cite{D} through which $23$ of the $59$ coefficients in the $N=5$ problem are determined from the results for $n=2,3,4$. 

Maple calculations suggest that further recursive rules involving other tetris type shapes will likely require "broken" shapes. As Dunne suggests \cite{D} it is very likely that such  rules exits. 





\end{document}